\documentclass[10pt]{amsart}
\usepackage{amssymb,amsbsy,amsmath,amsfonts,amssymb}
\usepackage{latexsym,euscript,exscale}

\title[Subspaces of spaces with unconditional bases]{Characterising subspaces of Banach spaces with a  Schauder basis having the shift property}
\author {Christian Rosendal}
\thanks{The author acknowledges the support of  NSF grant DMS 0901405}
\address{Department of Mathematics, Statistics, and Computer Science (M/C 249)\\
University of Illinois at Chicago\\
851 S. Morgan St.\\
Chicago, IL 60607-7045\\
USA}
\email{rosendal.math@gmail.com}
\urladdr{http://www.math.uic.edu/$~$rosendal}

\usepackage{fouriernc}

\newcommand {\B}{\mathbb B}

\newcommand {\N}{\mathbb N}

\newcommand{\eps}{\epsilon}

\newcommand{\tom} {\emptyset}

\newcommand{\saa}{\Rightarrow}
\newcommand{\equi}{\Leftrightarrow}

\newcommand{\til}{\rightarrow}

\newcommand{\Lim}[1]{\mathop{\longrightarrow}\limits_{#1}}

\newcommand {\del}{ \; \big| \;}

\newcommand {\e} {\exists}
\renewcommand {\a} {\forall}

\newcommand{\norm}[1]{\lVert#1\rVert}
\newcommand{\Norm}[1]{\big\lVert#1\big\rVert}

\newcommand{\triple}[1]{|\!|\!|#1|\!|\!|}
\newcommand{\Triple}[1]{\big|\!\big|\!\big|#1\big|\!\big|\!\big|}

\newtheorem{thm}{Theorem}

\newtheorem{lemme}[thm]{Lemma}
\newtheorem{prop} [thm] {Proposition}
\newtheorem{defi} [thm] {Definition}

\newtheorem{prob}[thm]{Problem}

\begin{document}
 \maketitle

\begin{abstract}
We give an intrinsic characterisation of the separable reflexive Banach spaces
that embed into separable reflexive spaces with an unconditional basis all of
whose normalised block sequences with the same growth rate are equivalent. This
uses methods of E. Odell and T. Schlumprecht.
\end{abstract}

\section{The shift property}
We consider in this paper a property of Schauder bases that has come up on several
occasions since the first construction of a truly non-classical Banach space
by B. S. Tsirelson in 1974 \cite{tsirelson}. It is a weakening of the property of  perfect homogeneity, which
replaces the condition 
\begin{quote}
{\em all normalised block bases are equivalent}
\end{quote}
with the weaker 
\begin{quote}
{\em all normalised block bases with the same growth rate are equivalent,}
\end{quote}
and is satisfied by bases constructed along the lines of the Tsirelson basis,
including the standard bases for the Tsirelson space and its dual.

To motivate our study and in order to fix ideas,  in the following result
we sum up a number of conditions that have been studied at various occasions in the literature and that can all be seen to be reformulations of the aforementioned property. Though I know of no single reference for the proof of the equivalence, parts of it are implicit in J. Lindenstrauss and L. Tzafriri's paper \cite{lindenstrauss} and the paper by P. G.  Casazza, W. B. Johnson and L. Tzafriri \cite{casazza}. Moreover, any idea needed for the proof can be found in, e.g., the book by F. Albiac and N. J. Kalton \cite{albiac} (see Lemma 9.4.1, Theorem 9.4.2. and Problem 9.1) and the statement should probably be considered folklore knowledge.
\begin{thm}\label{equivalences}
Let $(e_n)_{n=1}^\infty$ be a normalised unconditional Schauder basis for a Banach space $X$. Then the following conditions are equivalent.
\begin{enumerate}
\item Any block subspace is complemented.
\item Any block subspace $[x_n]_{n=1}^\infty$ is complemented by a projection $P$ such that 
$$
Pz=\sum_{n=1}^\infty x^*_n(z)x_n,
$$
where $x_n^*\in X^*$ satisfy ${\rm supp}\; x^*_n\subseteq {\rm supp}\; x_n$.
\item  If $(x_n)_{n=1}^\infty$ and $(y_n)_{n=1}^\infty$ are normalised block sequences of $(e_n)_{n=1}^\infty$ with
$$
x_1<y_1<x_2<y_2<\ldots,
$$
then $(x_n)_{n=1}^\infty\sim (y_n)_{n=1}^\infty$.
\item If $(x_n)_{n=1}^\infty$ is a normalised block basis, then $(x_{n})_{n=1}^\infty\sim(x_{n+1})_{n=1}^\infty$.
\item If $(x_i)_{i=1}^\infty$ and $(y_i)_{i=1}^\infty$ are normalised block sequences such that 
$$
\max({\rm supp}\; x_i\cup {\rm supp}\; y_i)<\min({\rm supp}\; x_{i+1}\cup {\rm supp}\; y_{i+1})
$$
for all $i$, then $(x_i)_{i=1}^\infty\sim(y_i)_{i=1}^\infty$.
\item For all normalised block bases $(x_n)_{n=1}^\infty$, if $k_n\in {\rm supp}\; x_n$ for all $n$,
then $(e_{k_n})_{n=1}^\infty\sim (x_n)_{n=1}^\infty$.
\end{enumerate}
Moreover, if any of the above properties hold, then they do so uniformly, e.g.,
in (4) there is a constant $C$ such that for all normalised block bases
$(x_n)_{n=1}^\infty$, we have $(x_n)_{n=1}^\infty\sim_C(x_{n+1})_{n=1}^\infty$. 
\end{thm}

An unconditional basis satisfying the above equivalent conditions will be said to have the {\em shift property}. This is a natural weakening of {\em perfect homogeneity}, i.e., that all normalised block bases are equivalent, which was shown to be just a reformulation of being equivalent to the standard unit vector basis of $c_0$ or $\ell_p$, $1\leqslant p<\infty$, by M. Zippin \cite{zippin}.  Let us also note that the shift property is stronger than what it is called the {\em block property} in \cite{kutzarova}, which is the requirement that every block sequence is equivalent with some subsequence of the basis. Finally, we remark that the shift property is obviously hereditary, that is, any normalised block basis of an unconditional basis with the shift property will itself have the shift property.

Moreover, while the canonical bases of both Tsirelson's space and its dual have the shift property, only one of them contains a {\em minimal} subspace, i.e., an infinite-dimensional subspace that embeds into all of its further infinite-dimensional subspaces. On the other hand, recall that a space $E$ is {\em locally minimal} \cite{minimal} if there is a constant $K$ such that for all finite-dimensional
$F\subseteq E$ and infinite-dimensional $X\subseteq E$, $F\sqsubseteq_KX$,
i.e., $F$ embeds with constant $K$ into $X$. As was pointed out in \cite{minimal} (Proposition 6.7), the proof of 
Theorem 14 in  \cite{casazza} essentially shows that any locally minimal space with a basis having the shift property is minimal.

The goal of the present paper is not to study the shift property per se, but rather to characterise the separable reflexive spaces that embed into a Banach space having a Schauder basis with the shift property. This will require some rather sophisticated techniques developed by E. Odell and T. Schlumprecht in a series of papers (see, e.g., \cite{haydon, odell}) and that we shall  summarise and slightly develop here. As a first application of their techniques, they characterised in \cite{odell} the separable reflexive Banach spaces embedding into an $\ell_p$-sum of finite-dimensional spaces for $1<p<\infty$ and their result was further improved in \cite{IAG} to the following statement.

\begin{thm}[see \cite{odell, IAG}]\label{OS}
Let $E$ be a separable reflexive Banach space such that any normalised weakly null tree $T$ in $E$ has a  branch $(x_i)_{i=1}^\infty\in [T]$ equivalent with all its subsequences. Then $E$ embeds into an $\ell_p$-sum, $1<p<\infty$, of finite-dimensional spaces.
\end{thm}

The result we shall obtain here has a weaker, though similar sounding hypothesis, but its conclusion is perhaps more satisfactory, since it provides a  basis rather than a finite-dimensional decomposition.

\begin{thm}\label{ros}
Let $E$ be a separable reflexive Banach space such that any normalised weakly null tree $T$ in $E$ has a  branch $(x_i)_{i=1}^\infty\in [T]$ satisfying $(x_{2i-1})_{i=1}^\infty\sim(x_{2i})_{i=1}^\infty$. Then $E$ embeds into a reflexive space having an unconditional basis with the shift property.
\end{thm} 

If the reader is not familiar with the techniques of Odell and Schlumprecht, this  should not be a hindrance to understanding the present construction, as we shall take certain of their technical results as black boxes that are directly applicable in our situation.

Without further introduction, let us commence the technical part of the paper by proving Theorem \ref{equivalences} for the record and the convenience of the reader.

\begin{proof}
The implication (1)$\saa$(2) follows directly from Lemma 9.4.1 in \cite{albiac}, so we shall not repeat the proof here.

(2)$\saa$(3): Suppose (2) holds and $(x_n)_{n=1}^\infty$ and $(y_n)_{n=1}^\infty$ are normalised block sequences satisfying
$$
x_1<y_1<x_2<y_2<\ldots.
$$
Assume that $a_n$ are scalars such that $\sum_{n=1}^\infty a_nx_n$ converges and choose $s_n>0$ converging to $0$ such that also 
$\sum_{n=1}^\infty  \frac{a_n}{s_n} x_n$ converges. Put $w_n=x_n+s_ny_n$ and find $w_n^*\in X^*$ such that ${\rm supp}\;w_n^*\subseteq {\rm supp}\;w_n$ and 
$$
Pz=\sum_{n=1}^\infty w_n^*(z)w_n
$$
defines a bounded projection onto $[w_n]_{n=1}^\infty$, whence $\sup\norm{w_n^*}<\infty$. Then
$$
P\big(\sum_{n=1}^\infty  \frac{a_n}{s_n}x_n\big)=\sum_{n=1}^\infty  \frac{a_n}{s_n}P(x_n)=\sum_{n=1}^\infty  \frac{a_n}{s_n}w_n^*(x_n)w_n=\sum_{n=1}^\infty  \frac{a_n}{s_n}w_n^*(x_n)(x_n+s_ny_n)
$$
and so the last series is norm convergent. By unconditionality, it follows that the series $\sum_{n=1}^\infty a_nw_n^*(x_n)y_n$ is norm convergent too. Thus, as 
$$
w_n^*(x_n)=w_n^*(w_n)-w^*_n(s_ny_n)=1-s_nw_n^*(y_n)\Lim{n\til \infty}1,
$$
using unconditionality again, we find that also $\sum_{n=1}^\infty a_ny_n$ is norm convergent. A symmetric argument shows that if  $\sum_{n=1}^\infty a_ny_n$ converges, then so does  $\sum_{n=1}^\infty a_nx_n$, whence $(x_n)_{n=1}^\infty$ and $(y_n)_{n=1}^\infty$ are equivalent.

(3)$\saa$(4): Assume that (3) holds and that $(x_n)_{n=1}^\infty$ is a normalised block sequence. Then using (3) 
$$
(x_{2n-1})_{n=1}^\infty\sim (x_{2n})_{n=1}^\infty\sim (x_{2n+1})_{n=1}^\infty.
$$
By unconditionality, it follows that the sequence $(x_n)_{n=1}^\infty$, which is the disjoint union of the sequences $(x_{2n-1})_{n=1}^\infty$ and  $(x_{2n})_{n=1}^\infty$, is equivalent to the sequence $(x_{n+1})_{n=1}^\infty$, which itself is the disjoint union of the sequences $(x_{2n})_{n=1}^\infty$ and  $(x_{2n+1})_{n=1}^\infty$.

(4)$\saa$(5): If $(x_i)_{i=1}^\infty$ and $(y_i)_{i=1}^\infty$ are normalised block sequences such that 
$$
\max({\rm supp}\; x_i\cup {\rm supp}\; y_i)<\min({\rm supp}\; x_{i+1}\cup {\rm supp}\; y_{i+1}),
$$
then both $x_1,y_2,x_3,y_4,\ldots$ and $x_2,y_3,x_4,y_5,\ldots$ are normalised block sequences, whence $(x_{2i-1})_{i=1}^\infty\sim(y_{2i})_{i=1}^\infty$ and $(x_{2i})_{i=1}^\infty\sim(y_{2i+1})_{i=1}^\infty$. By unconditionality, it follows that $(x_i)_{i=1}^\infty\sim(y_{i+1})_{i=1}^\infty\sim (y_i)_{i=1}^\infty$.

(5)$\saa$(6): Trivial.

(6)$\saa$(1): If (6) holds, then it does so uniformly, that is, there is a constant $C$ such that $(x_n)_{n=1}^\infty\sim_C(e_{k_{n}})_{n=1}^\infty$ whenever $(x_n)_{n=1}^\infty$ is a normalised block basis and $k_n\in {\rm supp}\;x_n$.
This can easily be seen, as otherwise one would be able to piece together finite bits of sequences  with worse and worse constants of equivalence to get a counter-example to (6). Let also $K_u$ be the constant of unconditionality of $(e_n)_{n=1}^\infty$.

Suppose $(x_n)_{n=1}^\infty$ is a normalised block sequence and let $I_1<I_2<I_3<\ldots$ be a partition of $\N$ into successive finite intervals such that ${\rm supp}\;x_n\subseteq I_n$. Find also norm $1$ functionals $x_n^*\in X^*$ such that ${\rm supp}\; x_n^*\subseteq {\rm supp}\; x_n$ and $x_n^*(x_n)=1$. We claim that
$$
P(z)=\sum_{n=1}^\infty x_n^*(z)x_n
$$
defines a projection of norm $\leqslant K_uC^2$ from $X$ onto $[x_n]_{n=1}^\infty$. To see this, suppose $z\in X$ and write $z=\sum_{n=1}^\infty a_nz_n$, where the $z_n$ are normalised block vectors such that ${\rm supp}\;z_n\subseteq I_n$. Modulo, perturbing $x_n$ and $z_n$ ever so slightly to get ${\rm supp}\; x_n=I_n={\rm supp}\; z_n$ and picking $k_n\in I_n$, we see that $(x_n)_{n=1}^\infty\sim_C(e_{k_n})_{n=1}^\infty\sim_C(z_n)_{n=1}^\infty$. So $\sum_{n=1}^\infty a_n x_n$ converges and, by unconditionality, so does $\sum_{n=1}^\infty x_n^*(z_n)a_n x_n=\sum_{n=1}^\infty x_n^*(z) x_n$. Therefore, $P$ is defined and satisfies
\[\begin{split}
\norm{P(z)}&=\norm{\sum_{n=1}^\infty x_n^*(z) x_n}=\norm{\sum_{n=1}^\infty x_n^*(z_n) a_nx_n}\\
&\leqslant K_u\norm{\sum_{n=1}^\infty a_nx_n}\leqslant K_uC^2\norm{\sum_{n=1}^\infty a_nz_n}=K_uC^2\norm{z},
\end{split}\]
proving the estimate on the norm.
\end{proof}

Finally, let us also remark that unconditionality is already implied by conditions (4), (5) and (6) of Theorem \ref{equivalences}. E.g., if a normalised basis  $(e_n)_{n=1}^\infty$ satisfies (4) and
$(\theta_n)_{n=1}^\infty\in\{-1,1\}^\infty$, then 
$$
(e_1,\theta_1
e_2,\theta_1\theta_2e_3,\theta_1\theta_2\theta_3e_4,\ldots)\sim (\theta_1
e_2,\theta_1\theta_2e_3,\theta_1\theta_2\theta_3e_4,\theta_1\theta_2\theta_3\theta_4e_5,\ldots),
$$
and therefore, multiplying both sides with $(\theta_1,\theta_1\theta_2,\theta_1\theta_2\theta_3,\theta_1\theta_2\theta_3\theta_4,\ldots)$, we have
$$
(\theta_1 e_1,\theta_2e_2,\theta_3e_3,\theta_4e_4,\ldots)\sim (e_2,e_3,e_4,e_5,\ldots)\sim(e_1,e_2,e_3,e_4,\ldots).
$$
Since $(\theta_n)_{n=1}^\infty\in\{-1,1\}^\infty$ is arbitrary, this shows that $(e_n)_{n=1}^\infty$ is unconditional.

Before continuing with the proof of Theorem \ref{ros}, let us note that, while Theorem \ref{ros} characterises reflexive spaces embeddable into a space with a basis having the shift property, we do not know of any significant characterisation of the spaces containing a basic sequence with the shift property. Using W. T. Gowers' block Ramsey theorem from \cite{gowers} and Lemma 6.4 of \cite{minimal}, we can conclude that if $X$ is a Banach space with a Schauder basis $(e_n)_{n=1}^\infty$, then $X$ contains a normalised block sequence $(y_n)_{n=1}^\infty$ that either is unconditional and has the shift property or such that there is  a non-empty tree $T$ consisting of finite normalised block sequences of $(y_n)_{n=1}^\infty$ with the following property:
\begin{itemize}
\item[(a)] if $(z_1,\ldots,z_m)\in T$ and $Z$ is a block subspace of $[y_n]_{n=1}^\infty$, then there is $z\in Z$ such that $(z_1,\ldots,z_m, z)\in T$, and
\item[(b)] if $(z_1,z_2,z_3,\ldots)$ is an infinite branch of $T$, then $(z_{2n-1})_{n=1}^\infty\not \sim(z_{2n})_{n=1}^\infty$.
\end{itemize}
However, it is not clear what can be concluded from the existence of such a tree $T$ and one would like to draw stronger or more informative consequences from this.
\begin{prob}
Formulate and prove a dichotomy that characterises the Banach spaces containing a unconditional basis sequence with the shift property.
\end{prob}


\section{Subspaces of spaces with an F.D.D.}\label{subspaces of
spaces with FDD's}

We fix in the following Banach spaces $E\subseteq F$ and an F.D.D. $(F_i)_{i=1}^\infty$ of $F$. For each interval $I\subseteq \N$, we let $I(x)$  denote the canonical projection of  $x\in F$
onto the subspace $\sum_{i\in I}F_i$ and shall also sometimes write
$[\sum_{i\in I}F_i](x)$ for $I(x)$ if there is any chance of confusion. So, if $K$ denotes the constant of the decomposition $(F_i)_{i=1}^\infty$, then $\norm{I}\leqslant 2K$ for any interval $I\subseteq \N$.

Fixing notation, if $A$ is a set, we let $A^\infty$ denote the set of all infinite sequences $(a_i)_{i=1}^\infty$ of elements of $A$ and let $A^{<\infty}$ denote the set of all finite sequences $(a_1,\ldots,a_n)$ of elements of $A$, including the empty sequence $\tom$. A {\em tree} on $A$ is a subset $T\subseteq A^{<\infty}$ closed under initial segments, i.e., such that  $(a_1,\ldots,a_n)\in T$ implies that $(a_1,\ldots,a_m)\in T$ for all $m\leqslant n$. When $T$ is a tree on $A$, we let $[T]$ denote the set of all {\em infinite branches} of $T$, i.e., the set of all sequences $(a_i)_{i=1}^\infty$ such that $(a_1,\ldots,a_{n})\in T$ for all $n$.

To simplify notation, if $\Delta=(\delta_i)_{i=1}^\infty$  
is a decreasing sequence of real numbers $\delta_i>0$ tending to $0$, we will denote this simply by $\Delta\searrow0$.
Similarly, if $M=(m_i)_{i=1}^\infty$ is a strictly increasing sequence of natural numbers, we shall denote this by $M\nearrow\infty$.

If $\B\subseteq S_E^\infty$ is a set of normalised sequences in $E$, we let 
$$
\B_\Delta=\big\{(x_i)_{i=1}^\infty\in S_E^\infty\del \e (y_i)_{i=1}^\infty\in \B\; \a i\; \norm{x_i-y_i}<\delta_i\big\}
$$
and
$$
{\rm Int}_\Delta(\B)=\big\{(x_i)_{i=1}^\infty\in S_E^\infty\del \a (y_i)_{i=1}^\infty\in S_E^\infty\; \big(\a i\; \norm{x_i-y_i}<\delta_i\til (y_i)_{i=1}^\infty\in  \B\big)\big\},
$$
and note that ${\rm Int}_\Delta(\B)=\;\sim\!(\sim \B)_\Delta$, where the complement is taken with respect to $S_E^\infty$.

\begin{defi}\label{delta-block}
Given $\Delta\searrow0$, a normalised sequence $(x_i)_{i=1}^\infty\in S_E^\infty$ is said to be
a $\Delta$-{\em block sequence} if there are intervals $I_i\subseteq \N$ such that
$$
I_1<I_2<I_3<\ldots
$$
and for every $i$,
$$
\|{I_i}(x_i)-x_i\|<\delta_i.
$$
Moreover, if $M\nearrow \infty$, we say that $(x_i)_{i=1}^\infty$ is {\em $M$-separated} if the witnesses $I_i\subseteq \N$ can be chosen such that 
$$
m_1<I_1\quad \&\quad \a i\;\e j\;\;\; I_i<m_j<m_{j+1}<I_{i+1}.
$$

We let ${\rm bb}_{E,\Delta}(F_i)$ denote the set of  $\Delta$-block sequences in $E$ and let ${\rm bb}_{E,\Delta, M}(F_i)$ denote the set of $M$-separated $\Delta$-block sequences in $E$.
\end{defi}

We notice that if $K$ is the constant of the decomposition $(F_i)_{i=1}^\infty$ and $(x_i)_{i=1}^\infty$ and $(y_i)_{i=1}^\infty$ are normalised sequences such that $ \|x_i-y_i\|<\delta_i$ for all $i$, then if $(x_i)_{i=1}^\infty$
is a $\Delta$-block, $(y_i)_{i=1}^\infty$ is a $4K\Delta$-block (with the same sequence of witnesses $I_1<I_2<\ldots$).

Also, since $\Delta\searrow0$ is a decreasing sequence, the sets  ${\rm bb}_{E,\Delta}(F_i)$ and  ${\rm bb}_{E,\Delta, M}(F_i)$ are closed under taking subsequences, that is, if $(x_i)_{i=1}^\infty\in {\rm bb}_{E,\Delta, M}(F_i)$, as witnessed by a sequence $(I_i)_{i=1}^\infty$, and $A\subseteq \N$, then $(I_i)_{i\in A}$ witnesses that $(x_i)_{i\in A}\in {\rm bb}_{E,\Delta, M}(F_i)$. Lemma \ref{all blocks} below essentially improves this to closure under taking normalised block sequences.

\begin{lemme}\label{interior}
Suppose $E$ is a subspace of a space $F$ with an F.D.D. $(F_i)_{i=1}^\infty$. Let
$\B\subseteq  S_E^\infty$ be a set of sequences invariant under equivalence.
Then there is a $\Delta\searrow0$ such that
$$
\B_{\Delta}\cap {\rm bb}_{E,\Delta}(F_i)\subseteq {\rm Int}_{\Delta}(\B).
$$
\end{lemme}

\begin{proof}Pick a $\Delta\searrow0$ depending on the constant of the decomposition $(F_i)_{i=1}^\infty$ such that if $(y_i)_{i=1}^\infty$ is a normalised block sequence in $F$ and $(v_i)_{i=1}^\infty$ is a sequence in $F$ satisfying $\norm{v_i-y_i}<5\delta_i$ for all $i$, then $(v_i)_{i=1}^\infty\sim (y_i)_{i=1}^\infty$. Assume also that $\delta_i>2\delta_i^2$ for every $i$.

Now, suppose $(x_i)_{i=1}^\infty\in\B_{\Delta}\cap {\rm bb}_{E,\Delta}(F_i)$ and let $(u_i)_{i=1}^\infty$  be a normalised sequence in $E$ such that $\norm{x_i-u_i}<\delta_i$ for all $i$. We must show that $(u_i)_{i=1}^\infty\in \B$, which will imply that $(x_i)_{i=1}^\infty\in {\rm Int}_\Delta(\B)$.

By assumption on $(x_i)_{i=1}^\infty$, we can find $(z_i)_{i=1}^\infty\in \B$ and intervals $I_1<I_1<\ldots$ such that $\norm{x_i-z_i}<\delta_i$  and $\norm {{I_i}(x_i)-x_i}<\delta_i$ for all $i$. Letting $y_i=\frac{{I_i}(x_i)}{\norm{{I_i}(x_i)}}$, we see that $(y_i)_{i=1}^\infty$ is a normalised block sequence in $F$ and a simple calculation using $\delta_i>2\delta_i^2$ gives $\norm{x_i-y_i}<4\delta_i$, whence $\norm{u_i-y_i}<5\delta_i$ and $\norm{z_i-y_i}<5\delta_i$. It follows that 
$(u_i)_{i=1}^\infty\sim (y_i)_{i=1}^\infty\sim (z_i)_{i=1}^\infty\in \B$ and so also $(u_i)_{i=1}^\infty\in \B$.
\end{proof}

\begin{lemme}\label{all blocks}
Suppose $E$ is a subspace of a space $F$ with an F.D.D. $(F_i)_{i=1}^\infty$ and $\Theta\searrow0$. Then there is $\Gamma\searrow0$ such that for any $M\nearrow0$ and $(x_i)_{i=1}^\infty\in  {\rm bb}_{E,\Gamma, M}(F_i)$,
\begin{enumerate}
\item $(x_i)_{i=1}^\infty$ is a normalised basic sequence, and
\item any normalised block sequence $(z_i)_{i=1}^\infty$ of $(x_i)_{i=1}^\infty$ belongs to ${\rm bb}_{E,\Theta,M}(F_i)$.
\end{enumerate}
\end{lemme}

\begin{proof}Let $K$ be the constant of the decomposition $(F_i)_{i=1}^\infty$. As in the proof of Lemma \ref{interior}, there is some $\Lambda\searrow0$ such that if $(x_i)_{i=1}^\infty\in {\rm bb}_{E,\Lambda}(F_i)$, as witnessed by a sequence of intervals $(I_i)_{i=1}^\infty$, then
$$
(x_i)_{i=1}^\infty\sim_2\Big(\frac{I_ix_i}{\norm{I_ix_i}}\Big)_{i=1}^\infty.
$$
Let now $\Gamma\searrow0$ be chosen such  that
$12K^2\sum_{i=m}^\infty\gamma_i<\theta_m$ and $\gamma_m<\lambda_m$ for all $m$.

Now suppose $(x_i)_{i=1}^\infty\in  {\rm bb}_{E,\Gamma, M}(F_i)$ for some $M\nearrow \infty$, as witnessed by a sequence of intervals
$(I_i)_{i=1}^\infty$. Then $(x_i)_{i=1}^\infty\in  {\rm bb}_{E,\Lambda}(F_i)$ and hence is $2$-equivalent to the normalised block basis $\Big(\frac{I_ix_i}{\norm{I_ix_i}}\Big)_{i=1}^\infty$, whence $(x_i)_{i=1}^\infty$ is itself a basic sequence.

Suppose also that $z=\sum_{i=n}^ma_ix_i$ is a block vector. We claim
that if we let $J=[\min I_n,\max I_m]$, then
$$
\norm{Jz-z}<\theta_n\norm{z},
$$
which is enough to obtain condition (2). To see this, notice first that for $i=n,\ldots m$,
\begin{displaymath}\begin{split}
\norm{Jx_i-x_i}
=&\Norm{[1,\min I_n-1](x_i)+[\max I_m+1,\infty[(x_i)}\\
=&\Norm{[1,\min I_n-1](x_i-I_ix_i)+[\max I_m+1,\infty[(x_i-I_ix_i)}\\
\leqslant&\Norm{[1,\min I_n-1](x_i-I_ix_i)}+\Norm{[\max I_m+1,\infty[(x_i-I_ix_i)}\\
\leqslant& K\norm{x_i-I_ix_i}+2K\norm{x_i-I_ix_i}\\
<&3K\gamma_i.
\end{split}\end{displaymath}
Since $\Norm{P_{I_i}}\leqslant 2K$ and $(x_i)_{i=1}^\infty$ is $2$-equivalent to $\Big(\frac{I_ix_i}{\norm{I_ix_i}}\Big)_{i=1}^\infty$, we have
$$
\sup_{n\leqslant i\leqslant m}|a_i|
=\sup_{n\leqslant i\leqslant m}\Norm{a_i\frac{I_ix_i}{\norm{I_ix_i}}}
\leqslant 2K\Norm{\sum_{i=n}^ma_i\frac{I_ix_i}{\norm{I_ix_i}}}
\leqslant 4K\Norm{\sum_{i=n}^ma_ix_i},
$$
and therefore
\begin{displaymath}\begin{split}
\Norm{J(\sum_{i=n}^ma_ix_i)-(\sum_{i=n}^ma_ix_i)}
=&\Norm{\sum_{i=n}^ma_i(Jx_i-x_i)}\\
\leqslant&\sum_{i=n}^m|a_i|\;\norm{Jx_i-x_i}\\
<& \sup_{n\leqslant i\leqslant m}|a_i|\cdot \sum_{i=n}^m3K\gamma_i\\
\leqslant& 12K^2\Norm{\sum_{i=n}^ma_ix_i}\sum_{i=n}^m\gamma_i\\
\leqslant&\theta_n\|\sum_{i=n}^ma_ix_i\|,
\end{split}\end{displaymath}
which shows that $\norm{Jz-z}<\theta_n\norm z$.
\end{proof}

\begin{defi}
Given $\Delta\searrow0$, a $\Delta$-{\em block tree} $T$ is a non-empty tree on $S_E$ such that for all
$(x_1,\ldots,x_{n})\in T$ the set
$$
\{y\in S_E\del (x_1,\ldots,x_{n},y)\in T\}
$$
can be written as $\{y_i\}_{i=0}^\infty$, where for each $i$ there is an interval $I_i\subseteq \N$ satisfying
\begin{itemize}
\item $\norm{I_iy_i-y_i}<\delta_{n+1}$,
\item $\min I_i\Lim{i\til \infty}\infty$.
\end{itemize}
\end{defi}

Now, an easy inductive construction shows that any $\Delta$-block tree $T$ contains a subtree $T'\subseteq T$ such that any infinite branch in $T'$ is a $\Delta$-block sequence, i.e., $[T']\subseteq {\rm bb}_{E,\Delta}(F_i)$.
So, without loss of generality, we can always assume that any $\Delta$-block tree satisfies this additional hypothesis.

We recall the following result from \cite{IAG}, which is proved using infinite-dimensional Ramsey theory. A similar statement for closed sets was proved earlier by Odell and Schlumprecht in \cite{odell}.

\begin{thm}\label{strategy for I}
Let $\B\subseteq S_E^\infty$ be a  coanalytic set. Then the following
are equivalent.
\begin{enumerate}
\item $\e \Delta\searrow0\;\;\; \e M\nearrow\infty\quad  {\rm bb}_{E,\Delta,M}(F_i)\subseteq   {\rm Int}_\Delta(\B)$,
\item $\e \Delta\searrow0$  such that any $\Delta$-block tree has a branch in ${\rm Int}_\Delta(\B)$.
\end{enumerate}
\end{thm}

\begin{defi}
A {\em weakly null tree} is a tree $T$ on $S_E$ such that, for any $(x_1,\ldots,x_{n})\in T$, the set
$$
\{y\in S_E\del (x_1,\ldots,x_{n},y)\in T\}
$$
can be written as $\{y_i\}_{i=1}^\infty$ for some weakly null sequence $(y_i)_{i=1}^\infty$.
\end{defi}

We recall also a statement from \cite{IAG} that sums up some of the elements of the construction of Odell and Schluprecht from \cite{odell}
that we shall use in the following.

\begin{prop}\label{wbjohnson}
Let $E$ be a separable reflexive Banach space. Then there is a reflexive Banach space
$F\supseteq E$ having an F.D.D. $(F_i)_{i=1}^\infty$  and a constant $c>1$ such
that whenever $\Delta\searrow0$ and $T$ is a $\Delta$-block
tree in $S_E$ with respect to $(F_i)_{i=1}^\infty$, there is a weakly null tree $S$ in $S_E$ such that
$$
[S]\subseteq [T]_{\Delta c}\quad\&\quad[T]\subseteq [S]_{\Delta c}.
$$
\end{prop}

We can now assemble the above results into the following general lemma.

\begin{lemme}\label{basic}
Suppose $E$ is a separable reflexive Banach space and $\B\subseteq S_E^\infty$ is a coanalytic set, invariant under equivalence, such that any weakly null tree on $S_E$ has a branch in $\B$. 
Then there are $\Gamma\searrow0$, $M\nearrow \infty$ and  a reflexive space  $F\supseteq E$ with an F.D.D. $(F_i)_{i=1}^\infty$ such that  any element of ${\rm bb}_{E,\Gamma, M}(F_i)$ is a  basic sequence all of whose normalised block sequences belong to  $\B$.
\end{lemme}

\begin{proof}
Pick first, by Proposition \ref{wbjohnson}, a space $F$
containing $E$ with a shrinking F.D.D. $(F_i)_{i=1}^\infty$ and a constant $c>1$ such that,
for any $\Delta\searrow0$ and $\Delta$-block tree $T$ in $E$, 
there is a weakly null tree $S$ in $E$ with
\begin{equation}\label{close}
[S]\subseteq [T]_{\Delta c}\quad\&\quad[T]\subseteq [S]_{\Delta c}.
\end{equation}
Choose also, by Lemma \ref{interior}, some  $\Delta\searrow0$ such that
$$
\B_{\Delta c}\cap {\rm bb}_{E,\Delta c}(F_i)\subseteq {\rm Int}_{\Delta c}(\B).
$$

We claim that any $\Delta$-block tree has a branch in ${\rm Int}_\Delta(\B)$. To see this, suppose
$T$ is a $\Delta$-block tree and assume without loss of generality that  $[T]\subseteq {\rm bb}_{E,\Delta}(F_i)\subseteq {\rm bb}_{E,\Delta c}(F_i)$.
Pick also a  weakly null tree $S$ satisfying (\ref{close}). Then, as $[S]\cap \B\neq \tom$, also 
$$
\tom\neq [T]\cap \B_{\Delta c}\subseteq[T]\cap  {\rm bb}_{E,\Delta c}(F_i)\cap \B_{\Delta c}\subseteq [T]\cap {\rm Int}_{\Delta c}(\B)\subseteq [T]\cap {\rm Int}_{\Delta}(\B),
$$
showing that $T$ has a branch in ${\rm Int}_\Delta(\B)$.

Applying Theorem \ref{strategy for I}, we find some $\Theta\searrow0$ and $M\nearrow\infty$ such that ${\rm bb}_{E,\Theta,M}(F_i)\subseteq   {\rm Int}_\Theta(\B)\subseteq \B$ and, applying Lemma \ref{all blocks}, the statement follows.
\end{proof}


\section{Killing the overlap}

The next proposition is Corollary 4.4 in \cite{odell}, except that condition (5) is not listed in the statement of the corollary. However, it can easily be gotten from the proof, provided that one chooses, in the notation of the paper, $\eps_i<\delta_i$.

\begin{prop}\label{overlap}
Suppose $F$ is a reflexive space with an F.D.D. $(H_i)_{i=1}^\infty$, $E\subseteq F$ is a subspace and $\Sigma\searrow0$. Then there are integers $0=a_0<a_1<\ldots$ such that for all $x\in S_E$ there are a sequence $(x_i)_{i=1}^\infty$ in $E$, a subset
$D\subseteq \N$ and numbers $a_{i-1}<b_i\leqslant a_i$, $b_0=0$, satisfying the following five conditions.
\begin{enumerate}
  \item $x=\sum_{i=1}^\infty x_i$,
  \item $\a i\notin D\; \norm{x_i}<\sigma_i$,
  \item $\a i\in D\;\Norm{[H_{b_{i-1}+1}\oplus\ldots\oplus H_{b_i-1}]x_i-x_i}<\sigma_i\|x_i\|$,
  \item $\a i\;\Norm{[H_{b_{i-1}+1}\oplus\ldots\oplus H_{b_i-1}]x-x_i}<\sigma_i$,
  \item $\a i\; \norm{H_{b_i}x}<\sigma_i$.
\end{enumerate}
\end{prop}

Combining Lemma \ref{basic} and Proposition \ref{overlap}, we are now in a position to prove our main result, Theorem \ref{ros}.

\begin{thm}
Suppose that $E$ is a separable reflexive Banach space such that any weakly
null tree in $E$ has a branch $(x_i)_{i=1}^\infty$ satisfying $(x_{2i-1})_{i=1}^\infty\sim(x_{2i})_{i=1}^\infty$. Then $E$
embeds into a reflexive space with an unconditional Schauder basis having the shift property.
\end{thm}

\begin{proof}Applying Lemma \ref{basic} to the set 
$$
\B=\big\{(x_i)_{i=1}^\infty\in S_E^\infty\del (x_{2i-1})_{i=1}^\infty\sim(x_{2i})_{i=1}^\infty\big\},
$$
we find $\Gamma\searrow0$, $M\nearrow \infty$ and a reflexive space  $F\supseteq E$ with an F.D.D. $(F_i)_{i=1}^\infty$ such that  any element of ${\rm bb}_{E,\Gamma, M}(F_i)$ is a  basic sequence all of whose normalised block sequences $(y_i)_{i=1}^\infty$ satisfy $(y_{2i-1})_{i=1}^\infty\sim(y_{2i})_{i=1}^\infty$.

We claim that there is a constant $C\geqslant 1$ such that $(y_{2i-1})_{i=1}^\infty\sim_C(y_{2i})_{i=1}^\infty$ for any such normalised block basis $(y_i)_{i=1}^\infty$. For if not, then,  by concatenating finite bits of sequences, we would be able to produce  some $(u_i)_{i=1}^\infty \in {\rm bb}_{E,\Gamma, M}(F_i)$ and a normalised block sequence $(y_i)_{i=1}^\infty$  of $(u_i)_{i=1}^\infty$ failing 
$(y_{2i-1})_{i=1}^\infty\sim(y_{2i})_{i=1}^\infty$, which is impossible.

Since it suffices to prove the conclusion of the theorem for a cofinite-dimensional subspace of $E$, by considering the cofinite-dimensional subspaces $F_{m_1+1}\oplus F_{m_1+2}\oplus F_{m_1+3}\oplus\ldots $ and $E\cap \big(F_{m_1+1}\oplus F_{m_1+2}\oplus F_{m_1+3}\oplus\ldots\big)$ of respectively $F$ and $E$, we can, without loss of generality, assume that $m_1=0$ and thus not worry about the initial offset by $m_1$ in the definition of $M$-separation (cf. Definition \ref{delta-block}).

Pick $(u_i)_{i=1}^\infty\in {\rm bb}_{E,\Gamma,M}(F_i)$. Then, for any choice of signs $\eps_i\in\{-1,1\}$, also $(\eps_iu_i)_{i=1}^\infty\in {\rm bb}_{E,\Gamma,M}(F_i)$ and hence $(\eps_{2i-1}u_{2i-1})_{i=1}^\infty\sim(\eps_{2i}u_{2i})_{i=1}^\infty$. It follows that $(u_{2i-1})_{i=1}^\infty$ is a basic sequence, equivalent to $(\eps_{2i}u_{2i})_{i=1}^\infty$ for any choice of signs $\eps_i\in \{-1,1\}$, and thus must be unconditional. 

Now $[u_{2i-1}]_{i=1}^\infty$ can be equivalently renormed so that $(u_{2i-1})_{i=1}^\infty$ is $1$-unconditional and, by a result of A. Pe\l czy\'nski \cite{pelczynski}, this renorming extends to an equivalent renorming of $F$. So, without loss of generality, we shall assume that $(u_{2i-1})_{i=1}^\infty$ is $1$-uncondi\-tio\-nal and has the shift property with some constant $C$.
Moreover, as $E$ is reflexive, it follows by a theorem of R. C. James (Theorem 3.2.13 in \cite{albiac}) that $(u_{2i-1})_{i=1}^\infty$ is both shrinking and boundedly complete.

We let $v_i=u_{2i+1}$, whence $(v_i)_{i=1}^\infty$ is the subsequence of $(u_{2i-1})_{i=1}^\infty$ omitting the first term.
Choose also $\sigma_i<\gamma_{2i-1}$ such that $\sum_{i=1}^\infty\sigma_i<\frac 1{24KC^2}$, , where $K$ denotes the constant of the decomposition $(F_i)_{i=1}^\infty$.

Since $(u_i)_{i=1}^\infty$ is an $M$-separated $\Gamma$-block sequence,  $\N$ can be partitioned into successive finite intervals
$$
L_1<I_1<R_1\;\;\;<\;\;\;L_2<I_2<R_2\;\;\;<\;\;\;L_3<I_3<R_3\;\;\;<\ldots
$$
such that
\begin{itemize}
\item[(a)] $\norm {I_i(v_i)-v_i}<\gamma_{2i+1}$,
\item[(b)] for every $i>1$ there is a $j$ such that $[m_j,m_{j+1}]\subseteq L_i$
\item[(c)] and for every $i$ there is a $j$ such that $[m_j,m_{j+1}]\subseteq R_i$.
\end{itemize}
Moreover, for 
$$
H_i=\sum_{j\in L_i\cup I_i\cup R_i}F_j,
$$
let $(a_i)_{i=0}^\infty$ be given as in Proposition \ref{overlap} and set
$$
A_i=H_{a_{i-1}+1}\oplus\ldots\oplus H_{a_i}.
$$
We define a new norm $\triple\cdot$ on ${\rm span}(\bigcup_{i=1}^\infty A_i)$ by setting
$$
\triple y=\Big\|\sum_{i=1}^\infty \|A_iy\|v_{a_i}\Big\|.
$$
Since $(v_i)_{i=1}^\infty$ is  $1$-unconditional,  $\triple\cdot$ is
indeed a norm and we can therefore consider the completion $V=\overline{\rm span}^{\triple\cdot}\big(\bigcup_{i=1}^\infty A_i\big)$.
Moreover, we claim that the mapping 
$$
T\colon x\in E\mapsto \sum_{i=1}^\infty A_ix\in V
$$
is a well-defined isomorphic embedding of $E$ into $V$.

To see this, suppose $x\in S_E$ is fixed and let $(x_i)_{i=1}^\infty$, $(b_i)_{i=0}^\infty$ and $D\subseteq \N$
be given as in Proposition \ref{overlap}. Let also
$$
B_i=H_{b_{i-1}+1}\oplus\ldots\oplus H_{b_i-1}.
$$
Then the decomposition $F=F_1\oplus F_2\oplus F_3\oplus\ldots$ blocks as
\[\begin{split}
F=&A_1\oplus A_2\oplus A_3\oplus\ldots\\
=&B_1\oplus H_{b_1}\oplus B_2\oplus H_{b_2}\oplus B_3\oplus H_{b_3}\oplus\ldots,
\end{split}\]
where, moreover, 
$$
A_i\;\subseteq\; B_i\oplus H_{b_i}\oplus B_{i+1}
$$
and, letting $A_0$ be the trivial space $\{0\}$,
$$
B_{i}\;\subseteq\; A_{i-1}\oplus A_{i}.
$$
It follows that with respect to the ordering of the original decomposition $(F_i)_{i=1}^\infty$, we have
\begin{equation}
B_1<L_{b_1}<I_{b_1}<R_{b_1}<B_2<L_{b_2}<I_{b_2}<R_{b_2}<B_3<\ldots.
\end{equation}
Now, by condition (4) of Proposition \ref{overlap},
$$
\big | \|B_ix\|-\|x_i\|\big|\leqslant \|B_ix-x_i\|<\sigma_i,
$$
and so, using condition (5) of Proposition \ref{overlap}, we have
\begin{displaymath}\begin{split}
\|A_ix\|
\leqslant&2K\Norm{[B_i\oplus H_{b_i}\oplus B_{i+1}]x}\\
\leqslant& 2K \big(\|B_ix\|+\norm{H_{b_i}x}+\|B_{i+1}x\|\big)\\
<& 2K \big(\|x_i\|+\|x_{i+1}\|+3\sigma_i\big).
\end{split}\end{displaymath}
Note also that
$$
\|x_i\|\leqslant \|B_ix\|+\sigma_i\leqslant 2K\|A_{i-1}x\|+2K\|A_ix\|+\sigma_i,
$$
and, by condition (3) of Proposition \ref{overlap}, for any $i\in D$, we have
$$
\Norm{B_ix_i-x_i}<\sigma_i\norm {x_i}<\gamma_{2i-1}\norm {x_i}.
$$
List now $D$ increasingly as $D=\{d_1,d_2,d_3,\ldots\}$ and note that, as $2i< 2b_{d_i}+1$,
$$
\norm{I_{b_{d_i}}(v_{b_{d_i}})-v_{b_{d_i}}}<\gamma_{2b_{d_i}+1}\leqslant  \gamma_{2i}.
$$
Therefore, by the ordering (2) above, we see that
$$
\Big(\frac{x_{d_1}}{\norm{x_{d_1}}},v_{b_{d_1}},\frac{x_{d_2}}{\norm{x_{d_2}}},v_{b_{d_2}},\frac{x_{d_3}}{\norm{x_{d_3}}},v_{b_{d_3}},\frac{x_{d_4}}{\norm{x_{d_4}}},\ldots\Big)
$$
is an $M$-separated $\Gamma$-block sequence, as witnessed by the sequence of interval projections
$$
B_{d_1},I_{b_{d_1}},B_{d_2},I_{b_{d_2}},B_3,I_{b_{d_3}},B_4,\ldots,
$$
and hence $\big(\frac{x_{i}}{\norm{x_{i}}}\big)_{i\in D}\sim_C\big(v_{b_i}\big)_{i\in D}$.
Furthermore, as $(v_i)_{i=1}^\infty$ has the shift property with constant $C$ and $b_1\leqslant a_1<b_2\leqslant a_2<\ldots$, we have 
\begin{equation}\label{equiv}
(v_{b_{i+1}})_{i=1}^\infty\sim_C(v_{b_i})_{i=1}^\infty\sim_C(v_{a_i})_{i=1}^\infty\sim_C(v_{a_{i+1}})_{i=1}^\infty
\end{equation}
and therefore $\big(\frac{x_{i}}{\norm{x_{i}}}\big)_{i\in D}\sim_{C^2}\big(v_{a_i}\big)_{i\in D}$. Since now $\sum_{i=1}^\infty x_i$ converges and $\sum_{i\notin D}\norm{x_i}<\sum_{i\notin D}\sigma_i<\infty$, it follows that also $\sum_{i=1}^\infty \norm{x_i}v_{a_i}$ and  $\sum_{i=1}^\infty \norm{x_{i+1}}v_{a_i}$ converge. Since $\norm{A_ix}\leqslant2K \big(\|x_i\|+\|x_{i+1}\|+3\sigma_i\big)$ and $(v_i)_{i=1}^\infty$ is unconditional,  we finally see that the sum $\sum_{i=1}^\infty \norm{A_ix}v_{a_i}$ converges and hence that $Tx=\sum_{i=1}^\infty A_ix\in V$ is well-defined.

By the same mode of reasoning, one verifies the following sequence of inequalities.
\begin{displaymath}\begin{split}
\|x\|
=&\Norm{\sum_{i=1}^\infty x_i}\\
\leqslant&\Big\|\sum_{i\in D}\|x_i\|\frac{x_i}{\|x_i\|}\Big\|+\Norm{\sum_{i\notin D} x_i}\\
\leqslant&C\Big\|\sum_{i\in D}\|x_i\|v_{b_i}\Big\|+\sum_{i\notin D}\|x_i\|\\
\leqslant&C^2\Big\|\sum_{i=1}^\infty \|x_i\|v_{a_i}\Big\|+\sum_{i\notin D}\sigma_i\\
\leqslant&C^2\Big\|\sum_{i=1}^\infty \big(2K\|A_{i-1}x\|+2K\|A_ix\|+\sigma_i\big)v_{a_i}\Big\|+\frac14\\
\leqslant&2KC^2\big\|\sum_{i=1}^\infty \|A_{i-1}x\|v_{a_{i}}\big\|+2KC^2\big\|\sum_{i=1}^\infty \|A_ix\|v_{a_i}\big\|+C^2\sum_{i=1}^\infty \sigma_i+\frac14\\
\leqslant&2KC^2(C+1)\big\|\sum_{i=1}^\infty \|A_ix\|v_{a_i}\big\|+\frac12\\
\leqslant&4KC^3\triple{Tx}+\frac12.
\end{split}\end{displaymath}
Thus, as $\|x\|-\frac 12=\frac12\|x\|$, we have $\|x\|\leqslant 8KC^3\triple{Tx}$. 
A similar argument shows that $\triple{Tx}\leqslant 5KC^3\norm{x}$, whereby $T$ is an isomorphic embedding of $E$ into $V$.

We shall now show how to embed $V$ into a space with a basis having the block property, which will finish the proof of the theorem.
First, to simplify notation, we let $w_i=v_{a_i}$. Fix also  $k_i\geqslant 1$ such that $A_i$ embeds
with constant $2$ into $Z_i=\ell_\infty^{k_i}$. Then $V$ clearly embeds with
constant $2$ into $Z=\sum_{i=1}^\infty Z_i$ equipped with the norm
$$
\triple{y}'=\Big\|\sum_{i=1}^\infty\|Z_iy\|w_i\Big\|.
$$
Moreover, since $(w_i)_{i=1}^\infty=(v_{a_i})_{i=1}^\infty$ is both shrinking and boundedly complete, $Z$ is
reflexive.

For each $i$, we let $(e_1^i,e_2^i,\ldots,e_{k_i}^i)$ be the standard unit
vector basis for $\ell_\infty^{k_i}$. Then
$$
(f_i)_{i=1}^\infty=(e_1^1,e_2^1,\ldots,e_{k_1}^1,e_1^2,e_2^2,\ldots,e_{k_2}^2,\ldots)
$$
is a $1$-unconditional basis for $Z$, which we claim has the block property.
To see this, suppose $(y_i)_{i=1}^\infty$ is a normalised
block sequence of $(f_i)_{i=1}^\infty$ and set $r_i=\min {\rm supp}\; y_i$ and
$$
i\in A\equi \e j\;\; y_i\in Z_j.
$$
Notice that for all $j$ there are at most two distinct $i\notin A$ such that
$Z_jy_i\neq 0$. We can therefore split $\sim A$ into two sets $B$ and $D$ such that
for all $j$ there is at most one $i$ from each of $B$ and $D$ such that
$Z_jy_i\neq 0$. By unconditionality, it is enough to show that $(y_i)_{i\in
A}\sim(f_{r_i})_{i\in A}$, $(y_i)_{i\in B}\sim(f_{r_i})_{i\in B}$ and
$(y_i)_{i\in D}\sim(f_{r_i})_{i\in D}$. Since the cases $B$ and $D$ are
similar, let us just do $A$ and $B$.

For each $i\in B$, let $n_i$ and $m_i$ be respectively the minimal and maximal $j$ such that $Z_jy_i\neq0$, whence
$y_i=Z_{n_i}y_i+\ldots+Z_{m_i}y_i$ and $n_i<m_i<n_j<m_j$ for $i<j$ in $B$. In particular, this means that if
$$
z_i=\sum_{j=n_i}^{m_i}\|Z_jy_i\|w_{j},
$$
then $(z_i)_{i\in B}$ is a block sequence of $(w_i)_{i=1}^\infty$ and
$$
\norm{z_i}=\Norm{\sum_{j=n_i}^{m_i}\|Z_jy_i\|w_{j}}=\triple{y_i}'=1.
$$ 
As $(w_i)_{i=1}^\infty$ has the shift property, this means that $(z_i)_{i\in B}\sim(w_{n_i})_{i\in B}\
\sim(f_{r_i})_{i\in B}$.
On the other hand, if $(\lambda_i)_{i=1}^\infty\in c_{00}$, then
\begin{displaymath}\begin{split}
\Triple{\sum_{i\in B}\lambda_iy_i}' 
=& \Big\|\sum_{i\in B}\sum_{j=n_i}^{m_i}\|Z_j\lambda_iy_i\|w_{j}\Big\|\\
=& \Big\|\sum_{i\in B}|\lambda_i|\sum_{j=n_i}^{m_i}\|Z_jy_i\|w_{j}\Big\|\\
=& \Big\|\sum_{i\in
B}|\lambda_i|z_i\Big\|.
\end{split}\end{displaymath}
Since $(z_i)_{i\in B}$ is unconditional, it follows that $(y_i)_{i\in
B}\sim(z_i)_{i\in B}\sim(f_{r_i})_{i\in B}$.

We now partition $A$ into finite sets $a_j$ by setting
$$
i\in a_j\equi y_i\in Z_j.
$$
Then for all $(\lambda_i)_{i=1}^\infty\in c_{00}$
\begin{displaymath}\begin{split}
\Triple{\sum_{i\in A}\lambda_iy_i}' 
=& \Big\|\sum_{j=1}^\infty\Norm{\sum_{i\in a_j}\lambda_iy_i}w_{j}\Big\|\\
=& \Big\|\sum_{j=1}^\infty\;\big(\sup_{i\in a_j}|\lambda_i|\big)w_{j}\Big\|\\
=& \Big\|\sum_{j=1}^\infty\Norm{\sum_{i\in a_j}\lambda_if_{r_i}}w_{j}\Big\|\\
=&\Triple{\sum_{i\in A}\lambda_if_{r_i}}'.
\end{split}\end{displaymath}
So $(y_i)_{i\in A}\sim(f_{r_i})_{i\in A}$, which finishes the proof.
\end{proof}


\end{document}